\documentclass[11pt]{amsart}
\usepackage{amsmath,amscd,amssymb}
\usepackage[mathscr]{eucal}
\usepackage[all]{xy}

\newtheorem{theorem}{Theorem}[section]
\newtheorem{lemma}[theorem]{Lemma}
\newtheorem{proposition}[theorem]{Proposition}
\newtheorem{corollary}[theorem]{Corollary}

\theoremstyle{definition}

\newtheorem{example}[theorem]{Example}

\theoremstyle{remark}
\newtheorem{remark}[theorem]{Remark}

\numberwithin{equation}{section}

\def\la{\lambda}

\def\e{\varepsilon}
\def\de{\delta}
\def\NN{{\mathbb N}}
\def\RR{{\mathbb R}}
\def\CC{{\mathbb C}}

\def\cT{{\mathcal T}}

\def\mD{{\mathcal D}}
\def\cD{{\mathcal D}}
\def\Re{{\rm Re}\,}
\def\Im{{\rm Im}\,}
\def\ker{{\rm Ker}\,}
\def\Int{{\rm Int}\,}

\def\conv{{\rm conv}\,}

\def\diag{{\rm diag}\,}
\def\dist{{\rm dist}\,}
\def\Relint{{\rm rInt}\,}

\def\diag{{\rm diag}\,}
\def\tr{{\rm tr}\,}

\begin{document}
\title[In search of convexity]{In search of convexity: diagonals and numerical ranges}

\author{V. M\"uller and Yu. Tomilov}

\address{Institute of Mathematics,
Czech Academy of Sciences,
Zitna 25,115 67  Prague,
 Czech Republic}
\email{muller@math.cas.cz}

\address{Institute of Mathematics, Polish Academy of Sciences,
\' Sniadeckich str.8, 00-656 Warsaw, Poland}
\email{ytomilov@impan.pl}

\date{\today}

\thanks{The first author has been supported by grant No. 20-31529X of GA CR and RVO:67985840. The second author was partially supported
by NCN grant UMO-2017/27/B/ST1/00078.}

\begin{abstract}
We show that the set of all possible constant diagonals of a bounded Hilbert space operator is always convex. This, in particular, answers an open question of J.-C. Bourin ($2003$). Moreover,
we show that the joint numerical range of a commuting operator tuple is, in general, not convex, which fills a gap in the literature.  We also prove that the Asplund-Ptak numerical range (which is convex for pairs of operators) is, in general, not convex for tuples of operators. 
\end{abstract}

\maketitle

\section{Introduction}
Let $H$ be a  separable (complex) Hilbert space with inner product $\langle \cdot,\cdot \rangle,$ and let $T$ be a bounded linear operator acting on $H$.
By the classical result of Hausdorff and Toeplitz, the numerical range $W(T)=\{\langle Tx,x\rangle: x\in H, \|x\|=1\}$ is always a convex set. This is one of the most important properties of the numerical range and 
convexity of various sets related to the numerical range is the basic issue in the theory, and
underlying many of its developments.

Unfortunately, the convexity of $W(T)$ fails if a single operator $T$  is replaced by a tuple 
 $\cT=(T_1,\dots,T_n)$ of bounded linear operators on $H$. It is well known that the joint numerical range 
$$
W(\cT):=W(T_1,\dots,T_n):=\{(\langle T_1x,x\rangle,\dots,\langle T_nx,x\rangle): x\in H,\|x\|=1\}
$$
is, in general, not convex for $n\ge 2$. Apparently, Hausdorff knew this already in $1918$; for a simple example 
see e.g. \cite[p. 138]{Bonsall}
or \cite{li-poon}.
However, $W(\cT)$ still has some traces of convexity. In particular, as shown in \cite{li-poon-II}, 
the set $W(\cT)$ is  star-shaped if $\dim H \ge \left[\frac{2n+1}{2}\right](2n+1)^2,$ 
where $[\cdot]$ stands for the integer part. Thus $W(\cT)$ is always star-shaped 
if $\dim H=\infty,$ \cite[Proposition 4.1]{li-poon-II}. (To relate $W(\cT)$ to the setting of \cite{li-poon-II}, one should identify 
$W(\cT)$ with $W(\Re T_1,\Im T_1,\dots,\Re T_n, \Im T_n) \subset \mathbb R^{2n}.$)

In this note we study the convexity of numerical ranges in three related situations, which surprisingly escaped the attention of experts. 
First, 
we show that the joint numerical range of commuting tuples is not necessarily convex. Such an example seems not to exist in the literature.
Next we demonstrate that the version of the numerical range introduced by E. Asplund and V. Ptak \cite{asplund}, which is convex for all pairs of Hilbert space operators, is in general not convex for triples of operators even if the operators commute.

Finally, we address the convexity of the set of constant diagonals of operators and operator tuples. The study of the structure for  diagonals of operators in infinite dimensions has a long history, and we refer to the recent survey \cite{LW19} for its finer details.
In the beginning of $2000$'s it received an impetus due to the works by R. Kadison and W. Arveson, and has attracted a considerable attention over the last years. 
For a good introduction into Kadison's theory one may consult \cite{Arveson06}, see also \cite{Arveson_USA}.
For a comprehensive account of the latest developments in this developing area of research, see again \cite{LW19}.

For $T \in B(H)$ acting on a separable space $H,$ its set of diagonals is defined as
$$\mathcal D(T):=\{(\langle T e_n, e_n \rangle)_{n=1}^{N}\}$$ when $(e_n)_{n=1}^{N}$ varies through all orthonormal bases of $H$ and $N =\dim H.$ 
While $\mathcal D(T)$ is rarely convex as a subset of $l^\infty$ (see Section \ref{diagonals}), we show that  
 the set of all $\la\in\CC$ such that $\langle Te_n, e_n \rangle= \la, 1 \le n \le N,$ for some orthonormal basis $(e_n)_{n=1}^N$ in $H$
 is always convex. This set is naturally identified with a subset $\mathcal D_{{\rm const}} (T)$ of $\mathcal D(T)$
consisting of constant diagonals.  The result gives a positive answer to a question of J.-C. Bourin from \cite[p. 213]{B}. In case of operator tuples, the convexity of $\mathcal D_{{\rm const}} (T_1, \dots, T_n)$ remains an open problem.

\section{Non-convexity of numerical ranges for commuting tuples}
Let $B(H)$ denote the space of all bounded linear operators on a Hilbert space $H,$ and
let $\cT=(T_1,\dots,T_n)\in B(H)^n$.
While the joint numerical range $W(\cT)$ is not in general convex, 
it can of course be convex for particular classes of $\cT,$
and moreover, one may define other useful joint numerical ranges associated to $\cT$
having sometimes better geometric properties.

The convexity of various types of (joint) numerical ranges has been studied intensively, see e.g. \cite{bolotnikov}, \cite{gutkin}, \cite{li-poon}, \cite{li-poon2}, and \cite{MT_JFA}-\cite{MT} and the references therein. 
In many of these results, the commutativity of the operators plays an important role. For example, it is well known that the joint numerical range of each commuting tuple of normal operators is convex 
(see e.g. \cite[Chapter 7.35, Theorem 5]{Bonsall} or \cite[Theorem 2.5]{Das_Gl}), while there are non-commuting tuples of selfadjoint operators with non-convex joint numerical range even in a two-dimensional space, see e.g. \cite[Example 1.1]{li-poon}. In \cite[Theorem 3.1]{bolotnikov}, the convexity of the joint numerical range of doubly commuting matrices was proved. It is also well known that spectral properties of commuting tuples are much better than those of non-commuting tuples. This allows one to show that if $\cT=(T_1, \dots, T_n) \in B(H)^n$ and $T_1, \dots, T_n$ commute, then
$$
\Int \conv \sigma (\cT) \subset W(\cT),
$$
where  $\Int \conv \sigma(\cT)$ stands for the interior of the convex hull 
of the spectrum $\sigma(\cT),$ see \cite[Corollary 4.3]{MT_LMS} for more details and proof.
 
So the joint numerical range  
of a commuting $n$-tuple $\cT=(T_1, \dots, T_n)$ exhibits some additional convexity properties, and there was some hope that it might be convex for all commuting tuples. Apparently, 
this question remained open for a long time. (To our knowledge, \cite[p. 522]{Das_PAMS} is the earliest reference where the question has been mentioned explicitly.)

The next example fills this gap and shows that the joint numerical range of commuting tuples is not convex, in general. 
The example appeared first in \cite{M}. Inspired by \cite{M}, very recently, another example of even two commuting matrices with non-convex numerical range was given in \cite{LLP}. Nevertheless, the present example has merit of being much simpler than the one in \cite{LLP}, and moreover, it can be applied to the situation of Asplund-Ptak numerical range without any changes, as we show below.

\begin{theorem}\label{commtuple}
For any Hilbert space $H, \dim H \ge 4,$ there exists a triple $\cT=(T_1, T_2, T_3)$ of  mutually commuting bounded operators on $H$, such that their joint numerical range
$W(\cT)$ is not convex.
\end{theorem}

\begin{proof}

First, let $H=\CC^4$ with the standard basis $e_1,e_2,e_3$ and $e_4$, and for $x \in \CC^4$ write $x=(x_1,x_2,x_3,x_4).$ Let
the linear operators $T_1, T_2$ and $T_3$ on $H$ be given by
$$
T_1=\begin{pmatrix}
0&0&1&0\cr
0&0&0&0\cr
0&0&0&0\cr
0&0&0&0
\end{pmatrix},\quad
T_2=\begin{pmatrix}
0&0&0&1\cr
0&0&0&0\cr
0&0&0&0\cr
0&0&0&0
\end{pmatrix}\quad\hbox{and}\quad
T_3=\begin{pmatrix}
0&0&0&0\cr
0&0&1&0\cr
0&0&0&0\cr
0&0&0&0
\end{pmatrix}.
$$
Let $H_0$ be the two-dimensional subspace spanned by $e_1$ and $e_2$. Clearly $\Im T_j\subset H_0$ and $\ker T_j\supset H_0$ for all $j=1,2,3$ (where $\Im T_j$ denotes the range and $\ker T_j$ the kernel of $T_j$, respectively). So we have $T_jT_k=0$ for all $j,k=1,2,3$.
In particular, the operators $T_1,T_2,$ and $T_3$ are mutually commuting.

We show that the numerical range of the triple $\cT=(T_1,T_2,T_3)$ is not convex. A direct computation shows that 
$$
W(\cT)=\bigl\{(x_3 \bar x_1, x_4\bar x_1, x_3 \bar x_2): x\in \CC^4, \|x\|_{\CC^4}=1\bigr \}.
$$
In particular, for $x_1=0, x_2=\sqrt{2}/2, x_3=\sqrt{2}/2, x_4=0$ we have 
$\alpha:=(0,0,1/2)\in W(\cT)$. Similarly, for $x_1=\sqrt{2}/2,x_2=x_3=0,x_4=\sqrt{2}/2$ we have $\beta:=(0,1/2,0)\in W(\cT)$.

We show that the midpoint $\frac{\alpha+\beta}{2}:=(0,1/4,1/4)$ does not belong to $W(\cT)$. Suppose on the contrary that there exist
$x \in \mathbb C^4$ with $\|x\|_{\CC^4}=1$ such that 
$x_3\bar x_1=0$, $x_4\bar x_1=1/4$ and $x_3\bar x_2=1/4$. So either $x_3=0$ or $x_1=0$. If $x_3=0$ then $x_3\bar x_2=0$, a contradiction. If $x_1=0$ then $x_4\bar x_1=0$, a contradiction again. Therefore $\frac{\alpha+\beta}{2}\notin W(\cT),$ and $W(\cT)$ is not convex.

Note that the argument above also shows that if $a$ and $b$ are complex numbers 
such that $a\ne 0\ne b,$ then  $(0,a,b) \notin W(\mathcal T)$. 
(It is enough to repeat the same reasoning for any $(0,a,b)$ instead of the mid-point $(0,1/4,1/4).$)
This observation can be used to show that a triple of commuting operators with non-convex numerical range exist in any Hilbert space with dimension greater than $4$.

Indeed, let now $F$ be any nontrivial Hilbert space, and consider the operators $T_1'=T_1\oplus 0_F,$ $T_2'=T_2 \oplus 0_F,$ and $T_3'=T_3\oplus 0_F$
on a Hilbert space $H=\mathbb C^4\oplus F.$ Then clearly $T_1',$ $T_2'$ and $T_3'$ commute,
 and for $\cT'=(T_1', T_2', T_3')$ the points $\alpha:=(0,0,1/2)$ and $\beta:=
(0,1/2,0)$ belong to $W(\cT').$
On the other hand, for every $h =(x, f)\in H, \|h\|=1,$ we have
$$
(\langle T_1' h, h \rangle, \langle T_2' h, h \rangle, \langle T_3' h, h \rangle) \in \|x\|^2_{\mathbb C^4} W(\cT).
$$ 
So, using the observation above, one infers that $\frac{\alpha+\beta}{2}\notin W(\cT),$ and gets a contradiction again.
(It is instructive to note that, in fact, 
$W(\cT')=\{(\langle T_1x,x\rangle,\langle T_2x,x \rangle,\langle T_3x,x\rangle): x\in \mathbb C^4, \|x\|_{\mathbb C^4}\le 1\}).$
\end{proof}

In \cite{asplund}, E. Asplund and V. Ptak considered another type of numerical range. For $\cT=(T_1,\dots,T_n)\in B(H)^n$ 
define
$$
 W_{{\rm AP}}(\cT)=\bigl\{(\langle T_1x,y\rangle,\dots,\langle T_nx,y\rangle): x,y\in H, \|x\|\le 1, \|y\|\le 1\bigr\}.
$$
It was proved in \cite{asplund} that $ W_{{\rm AP}}(T_1,T_2)$ is convex for each pair  $(T_1,T_2)$. 

In fact the matrices $T_1,T_2$ and $T_3$ constructed in Theorem \ref{commtuple} can be used to show that, in general, $W_{{\rm AP}}(\cT)$ is not convex, even for $\cT=(T_1, T_2, T_3)$ with mutually commuting operators $T_1,T_2$ and $T_3$.

\begin{theorem}
For any Hilbert space $H, \dim H \ge 4,$ there exists a triple $\cT=(T_1, T_2, T_3)$ of mutually commuting bounded operators on $H$, such that 
$W_{{\rm AP}}(\cT)$ is not convex.
\end{theorem}

\begin{proof}
The proof is analogous to the proof of Theorem \ref{commtuple}. 
If $H=\mathbb C^4$ and $T_1,T_2$ and $T_3$ are the operators on $H$ defined in this proof,
then for $\cT=(T_1, T_2, T_3)$ one has
\begin{align*}
 W_{{\rm AP}}(\cT)&=\bigl\{(x_3\bar y_1, x_4 \bar y_1, x_3 \bar y_2): x, y \in\CC^4, \|x\|_{\mathbb C^4}\le 1, \|y\|_{\mathbb C^4}\le 1
\bigr \},
\end{align*} 
and one checks as before that
$(0,0,1/2)\in W_{{\rm AP}}(\cT)$, $(0,1/2,0)\in W_{{\rm AP}}(\cT)$ and $(0,1/4,1/4)\notin  W_{{\rm AP}}(\cT)$. So
the numerical range $\ W_{{\rm AP}}(\cT)$ is not convex. The general case can be considered precisely as in the proof
of Theorem \ref{commtuple}, and we omit easy details.
\end{proof}

\section{Convexity of the set of constant diagonals}

Let $H$ be a separable Hilbert space with $\dim H=N, 1 \le N \le \infty$. 
For an $n$-tuple $\cT=(T_1,\dots,T_n)\in B(H)^n$ denote by $\mD_{{\rm const}}(\cT)$ the set of all $n$-tuples $(\la_1,\dots,\la_n)\in\CC^n$ such that $\cT$ has constant diagonal $(\la_1,\dots,\la_n)$, i.e., there exists an orthonormal basis $(u_j)_{j=1}^N$ in $H$  with
$$
\langle T_ku_j,u_j\rangle=\la_k
$$
for all $j$,$ 1\le j\le N,$  and $k=1,\dots,n$. 

If the space $H$ is finite dimensional and $T\in B(H)$, then it is easy to see that $\mD_{{\rm const}}(T)$ is a singleton --- the normalized trace of $T$.
This is a classical result due to A. Parker, see \cite[Theorem 1.3.4 and p. 28]{Horn} or \cite{Fil69}.
We give its proof for completeness, and since the argument is instructive for our subsequent considerations.
For $T\in B(H)$ denote by $\tr T$ the trace of $T,$ whenever it is well defined.
\begin{proposition}
Let $H$ be a Hilbert space, $\dim H=N<\infty,$ and let $T\in B(H)$. Then $\mD_{{\rm const}}(T)=\{N^{-1}\tr T\}$.
\end{proposition}

\begin{proof}
If $\la\in\mD_{{\rm const}}(T)$ then $\tr T=N\la$, and so $\la=N^{-1}\tr T$.

Conversely, let $\la=N^{-1}\tr T$ and $(u_j)_{j=1}^N$ be any orthonormal basis in $H$. Then
$$
\la=N^{-1}\sum_{j=1}^N\langle Tu_j,u_j\rangle\in W(T)
$$
since $W(T)$ is convex. Let $u\in H$ be a unit vector such that $\langle Tu,u\rangle=\la$. 
Decomposing $H$ as  $H=\CC u\oplus \{u\}^\perp$, one infers that $T$ is of the form
$$
T=\begin{pmatrix}\la&*\cr
*&T'\end{pmatrix},
$$
where $\tr T'=(N-1)\la$. Hence, the induction on the dimension of $H$ yields a constant diagonal for $T$ equal to $\la$.
\end{proof}

\begin{corollary}\label{corol}
Let $H$ be a Hilbert space, $\dim H=N<\infty,$ and let $\cT=(T_1,\dots,T_n)\in B(H)^n$. Then the set $\mD_{{\rm const}}(\cT)$ is either a singleton
$\{(N^{-1}\tr T_1,\dots,N^{-1}\tr T_n)\}$, or it is empty. Hence $\mD_{{\rm const}}(\cT)$ is convex.
\end{corollary}

The next example shows that $\mD_{{\rm const}}(T_1, T_2)$ may be empty even for pairs of operators $(T_1,T_2)$ in a finite-dimensional space.

\begin{example}\label{e}
Let $T_1, T_2 \in B(\mathbb C^2)$ be given by 
$$
T_1=\begin{pmatrix}0&0\cr1&0\end{pmatrix}\qquad\hbox{and}\qquad T_2=\begin{pmatrix}1&0\cr 0&-1\end{pmatrix}.
$$
Then $\tr T_1=\tr T_2=0$. We show that $(0,0)\notin W(T_1,T_2)$. Consequently, $(0,0)\notin\mD_{{\rm const}}(T_1,T_2)$ and $\mD_{{\rm const}}(T_1,T_2)=\emptyset$.

Suppose on the contrary that $(0,0)\in W(T_1,T_2)$. So there exist $x_1, x_2\in\CC$ with $|x_1|^2+|x_2|^2=1$ such that
$x=\begin{pmatrix}x_1 \cr x_2\end{pmatrix}$ satisfies $\langle T_1x,x\rangle=\langle T_2x,x\rangle=0$.
We have
$$
\langle T_1x,x\rangle=x_1 \bar x_2,
$$
so either $x_1=0$ or $x_2=0$. If $x_1=0$ then $|x_2|=1$ and
$$
\langle T_2x,x\rangle=|x_1|^2-|x_2|^2=-1\ne 0,
$$
a contradiction. Similarly, if $x_2=0$ then $|x_1|=1$ and $$\langle T_2x,x\rangle=|x_1|^2-|x_2|^2=1,$$
 a contradiction again.
Hence $(0,0)\notin W(T_1,T_2)$.
\end{example}

The situation is much more involved in infinite-dimensional spaces. Let $H$ be a separable infinite-dimensional Hilbert space, and $\cT=(T_1,\dots,T_n)\in B(H)^n$. Recall that the joint essential numerical range $W_e(\cT)$ of $\cT$ is defined as the set of all $n$-tuples $(\la_1,\dots,\la_n)\in\CC^n$ such that there exists an orthonormal sequence $(u_j)_{j=1}^\infty$ in $H$ satisfying
$$
\lim_{j\to\infty}\langle T_ku_j,u_j\rangle=\la_k
$$
for all $k=1,\dots,n$.  Clearly,  $W_e(\cT)\subset\overline{W(\cT)}.$ An important property of $W_e(\cT)$ is that it is always non-empty, closed and convex, see \cite[Lemma 3.1]{BFT} or \cite{li-poon}.
 Moreover, by \cite[Theorem 3.1]{li-poon2},  each point of $W_e(\cT)$ is a star-center for the star-shaped set $\overline{W(\cT)}.$ It is
also crucial to note that
\begin{equation}\label{lancaster}
\conv \left(\overline{W(\cT)}\right) = \conv \left(W(\cT)\cup W_e(\cT)\right),
\end{equation}
see \cite[Theorem 5.1]{MT_LMS} for a simple proof, or alternatively, \cite{Takaguchi} or \cite[Theorem 5.2]{li-poon2}.
The set $W_e(\cT)$ is invariant under compact perturbations of $\cT$, and, in particular, 
 for any finite-rank projection $P$ on $H,$ 
\begin{equation}\label{proj} 
W_e(\cT)=W_e\bigl((I-P)T_1(I-P), \dots, (I-P)T_n(I-P)\bigr).
\end{equation}
This fact is very useful in various inductive arguments.

Recall that
\begin{equation}\label{interior}
\Int W_e(\cT)\subset\mD_{{\rm const}}(\cT)\subset W_e(\cT),
\end{equation}
thus $\mD_{{\rm const}}(\cT)$ is a union of $\Int W_e(\cT)$ and a part of $\partial W_e(\cT).$
The second inclusion in \eqref{interior} follows from the definition. The first inclusion is non-trivial and follows from \cite[Corollary 4.2]{MT}, see also \cite[Theorem 1.1]{MT} (and \cite[Theorem 1.2]{B} and \cite[Theorem 1(i)]{H} for $n=1$). Moreover, by  \cite[Corollary 4.2]{MT}, one can replace in \eqref{interior} the interior $W_{e}(\cT)$ by the relative interior of $W_e(\cT)$ in the smallest affine 
 subspace containing $W(\cT)$ (i.e. in the affine hull of $W(\cT)$). This is relevant, if e.g. $T_i, 1 \le i \le n,$ are selfadjoint. 
For other relations between $\mathcal D_{{\rm const}}(\cT)$ and $W_e(\cT)$ see \cite[Propositions 5.3 and 5.4]{MT}.

Since both $W_e(\cT)$ and $\Int W_e(\cT)$ are convex sets, the inclusions \eqref{interior}  suggest that 
$\mD_{{\rm const}}(\cT)$ is always close to a convex set. So it is reasonable to ask whether $\mD_{{\rm const}}(\cT)$ is convex itself. 
Recall that the sets $S \subset \mathbb C^n$ satisfying $\Int C \subset S \subset C$ for a convex set $C \subset \mathbb C^n$ are called \emph{almost convex} in the literature. They share some properties of convex sets, such as e.g. separation properties. For interesting spectral conditions for almost convexity of joint numerical ranges as well as a pertinent discussion of almost convex sets, see \cite{Ma1}, \cite{Ma2} and \cite{Ma3}.

In Theorem \ref{convex_diagon} below, we give a positive answer to this question for $n=1$. This solves a problem posed by  J.-C. Bourin in \cite[p. 213]{B}. Our proof is based on the following criterion for existence of zero diagonals due to P. Fan, \cite[Theorem 1]{F}.
The criterion has a ``Tauberian'' character expressing the property $0\in\mD_{{\rm const}}(T)$) for $T \in B(H)$ in terms of the limit behavior of  partial sums of diagonal entries of $T$. 

\begin{theorem}\label{FanT}
Let $H$ be a separable Hilbert space, $\dim H=\infty,$ and let $T\in B(H)$. 
Then $0\in\mD_{{\rm const}}(T)$ if and only if there exists an orthonormal basis $(u_j)_{j=1}^\infty\subset H$ such that
 the sequence 
$\left\{\sum_{j=1}^k\langle Tu_j,u_j\rangle: k \ge 1\right \}$
has a subsequence converging to zero.
\end{theorem}
It is worth to mention that the original proof of Theorem \ref{FanT} in \cite{F} contained a gap, 
which was recently corrected in \cite[Appendix B]{L}.

The proof of convexity for 
$\mathcal D_{{\rm const}}(T)$ is based on two lemmas. The first one
addresses the continuity of the Gram-Schmidt procedure, and it is surely known. However, we were not able
to find an appropriate reference.

\begin{lemma}\label{gram-schmidt}
Let $H$ be a Hilbert space. For all $m \in \mathbb N$  and $\eta\in(0,1)$ there exists $\delta_{m,\eta}>0$ with the following property: 
if $\{e_j: 1 \le j \le m\}\subset H$ is an orthonormal system of vectors  and vectors $e_1',\dots,e_m'\in H$ satisfy
$\max_{1 \le j \le m}\|e'_j-e_j\|\le\delta_{m,\eta},$ 
then there exists an orthonormal basis  $\{f_j: 1\le j \le m\}$ in $\bigvee_{j=1}^m e_j'$ 
such that
$$
\max_{1 \le j \le m}\|f_j-e_j\|\le\eta.
$$
\end{lemma}

\begin{proof}

We prove the statement by induction on $m$.
The statement is clear for $m=1$: set $\delta_{1,\eta}=\eta/2$. If $\|e_1'-e_1\|\le\delta_{1,\eta}$ then let $f_1=\frac{e_1'}{\|e'_1\|}.$ We have  
$$
\|f_1-e_1\|\le
\|f_1-e'_1\|+\|e'_1-e_1\|\le
\bigl|1-\|e'_1\|\bigr|+\delta_{1,\eta}\le
2\delta_{1,\eta}=\eta.
$$

Let $m\ge 2$ and suppose that the statement is true for $m-1$. 
Let $\eta'=\frac{\eta}{2m}$. Fix  $\delta_{m,\eta}>0$ such that $\delta_{m,\eta}\le\delta_{m-1,\eta'}$ and
$$
2(m-1)\bigl(\delta_{m,\eta}+(1+\delta_{m,\eta})\eta'\bigr)+2\delta_{m,\eta} \le\eta.
$$

Let $e_1,\dots,e_m$ be orthonormal vectors in $H$ and let $e'_1,\dots,e'_m$ be any vectors in $H$ satisfying $\max_{1\le j\le m}\|e_j'-e_j\|\le\delta_{m,\eta}$.

By the induction assumption, there exist orthonormal vectors $f_1,\dots,f_{m-1}\in\bigvee_{j=1}^{m-1}e'_j$ such that $\max_{1 \le j \le m-1}\|f_j-e_j\|\le\eta'$.
Set 
$$
\tilde f_m=e'_m-\sum_{j=1}^{m-1}\langle e'_m,f_j\rangle f_j.
$$
We have
$$
\|e'_m\|\le \|e_m\|+\|e'_m-e_m\|\le 1+\delta_{m,\eta},
$$
and for $1\le j \le m-1,$ 
\begin{align*}
|\langle e'_m,f_j\rangle|
\le
|\langle e'_m,e_j\rangle|+&|\langle e'_m,f_j-e_j\rangle|\\
\le&
\|e'_m-e_m\|+\|e'_m\|\cdot \|f_j-e_j\|\le \delta_{m,\eta}+(1+\delta_{m,\eta})\eta'.
\end{align*}
Thus,
$$
\|\tilde f_m-e'_m\|\le (m-1)\bigl(\delta_{m,\eta}+(1+\delta_{m,\eta})\eta'\bigr)
$$
and
$$
\bigl|1-\|\tilde f_m\|\bigr|\le
\|\tilde f_m-e'_m\|+\|e'_m-e_m\|\le (m-1)\bigl(\delta_{m,\eta}+(1+\delta_{m,\eta})\eta'\bigr)+\delta_{m,\eta}.
$$
Note that by the choice of $\delta_{m,\eta}$ we have $\tilde f_m \ne 0,$ and set $f_m=\frac{\tilde f_m}{\|\tilde f_m\|}$. Then, by construction, the vectors $f_1,\dots f_m$ are orthonormal,  $\bigvee_{j=1}^m f_j = \bigvee_{j=1}^m e'_j,$  
and moreover
\begin{align*}
\|f_m-e_m\|\le&
\|f_m-\tilde f_m\|+\|\tilde f_m-e'_m\|+\|e'_m-e_m\|\\
\le&
2(m-1)\bigl(\delta_{m,\eta}+(1+\delta_{m,\eta})\eta'\bigr)+2\delta_{m,\eta}\le\eta.
\end{align*}
\end{proof}
\begin{remark} Note that the set $\{e'_j: 1 \le j \le m\}$ may be at a positive distance from $\bigvee_{j=1}^m e_j.$ A posteriori, due to our choice of $\delta,$ it consists
of linearly independent vectors, and $\dim (\vee_{j=1}^m e_j')=m$.
\end{remark}
\begin{remark}
A different proof of Lemma \ref{gram-schmidt} was proposed by the referees.
Following their argument, one notes that the set ${\rm LI}(H^m)$ of linearly independent
$m$-tuples of elements from $H$ is open in $H^m$ (with the product topology),
and the set ${\rm ON}(H^m)$  of orthonormal $m$-tuples is closed in $H^m.$
Then using the determinant formulation
of the Gram–Schmidt process, one infers that 
the process is a
retract
of ${\rm LI}(H^m)$ onto ${\rm ON}(H^m).$
\end{remark}
The second, approximation lemma allows one to reduce the
convexity property of $\mathcal D_{{\rm const}}(T)$ to (essentially) Theorem \ref{FanT}.

\begin{lemma}\label{induction}
Let $H$ be a separable Hilbert space, $\dim H=\infty,$ and let $T\in B(H).$ Suppose there exist $\alpha ,\beta\in\RR$, $\alpha <0<\beta$, satisfying $\alpha \in{\mathcal D}_{{\rm {\rm const}}}(T)$ and $\beta\in W_e(T)$.
Then for every subspace $M\subset H$ with $\dim M <\infty$ and every $\e>0,$  there exists a subspace $M'\subset H$ such that $M\subset M', \dim M'<\infty,$ and 
$$
|\tr(P_{M'}TP_{M'})|\le \e, 
$$
where $P_{M'}$ denotes the orthogonal projection onto $M'.$
\end{lemma}

\begin{proof}
By the assumption, there exists an orthonormal basis $(u_j)_{j=1}^\infty$ in $H$ such that
$\langle Tu_j,u_j\rangle=\alpha$ for all $j\in\NN$.

Let $\dim M=m,$   and let $e_1,\dots,e_m$ be an orthonormal basis in $M$. 
For fixed $\varepsilon >0,$ let $\eta=\frac{\e}{4m\|T\|}$.
 
Find $k\in\NN$ so large that  
\begin{equation}\label{dist}
\dist\Bigl\{e_j, \bigvee_{j=1}^k u_j\Bigr\}<\de, \qquad j=1,\dots,m,
\end{equation}
where $\de=\delta_{m,\eta}$ is the number given by Lemma \ref{gram-schmidt}.

Let  $L=\bigvee_{j=1}^k u_j.$  Clearly $\|P_Le_j-e_j\|\le\de$ for all $j=1,\dots,m$. Let
$$\widetilde M:= P_L M=\bigvee_{j=1}^m P_L e_j.$$ 
By applying Lemma \ref{gram-schmidt} to the set $\{e_j: 1\le j \le m\}$ and its ``perturbation'' $\{e_j':=P_L e_j: 1 \le j \le m \},$ we infer that
there exists an orthonormal basis $\{f_j: 1\le j \le m \}\subset \widetilde M$
such that 
$$
\max_{1\le j \le m} \|f_j -  e_j \| \le \frac{\varepsilon}{4m\|T\|}.
$$

Note that for any $x \in L$ and $1 \le j \le m$ one has 
$$\langle x, e_j\rangle=\langle P_L x, e_j\rangle=\langle x, P_L e_j \rangle,$$
 so that $x \in M^\perp$ if and only if $x \in \widetilde M^\perp,$ that is $L\cap M^\perp=L\cap \widetilde M^\perp$. Hence $L=\widetilde M\oplus (L\cap M^\perp)$. Let $$K=M\oplus(L\cap M^\perp).$$
We have
$$
\tr (P_{K}TP_{K})=
\tr (P_{M}TP_{M})+\tr (P_{L\cap M^\perp}TP_{L\cap M^\perp}),
$$
and
$$
\alpha k=\tr(P_LTP_L)=\tr (P_{\widetilde M}TP_{\widetilde M})+\tr (P_{L\cap M^\perp}TP_{L\cap M^\perp}).
$$
Thus
\begin{align*}
\bigl|\tr (P_{K}TP_{K})-\alpha k\bigr|
=&\bigl|\tr (P_{\widetilde M}TP_{\widetilde M})-\tr (P_{M}TP_{M})\bigr|\\
\le& 
\sum_{j=1}^m \bigl|\langle Tf_j,f_j\rangle-\langle Te_j,e_j\rangle\bigr|\\
\le& \sum_{j=1}^m 2\|T\|\cdot\|f_j-e_j\|\le\frac{\e}{2}.
\end{align*}

Recalling the notation $[\cdot]$ for the integer part, set
$$ 
n:=\Bigl[\frac{|\alpha k|}{\beta}\Bigr] \qquad \text{and}\qquad  \gamma=|\alpha k|-n\beta.
$$
Then $0\le \gamma  <\beta$ and so $\gamma \in [\alpha ,\beta]\subset W_e(T)$ in view of convexity of $W_e(T).$ Moreover,
$$\alpha k+n\beta+\gamma=0$$
by the choice of $n$ and $\gamma.$
Using \eqref{proj}, choose inductively orthonormal vectors $x_1,\dots,x_{n+1}\in K^\perp$ such that
$$ |\langle Tx_j,x_j\rangle-\beta|<\frac{\e}{2(n+1)}, \quad 1\le j \le n, \quad \text{and}\quad
|\langle Tx_{n+1},x_{n+1}\rangle-\gamma|<\frac{\e}{2(n+1)}.$$
Let $M'=K\oplus\bigvee_{j=1}^{n+1}x_j$. Then $M\subset K\subset M'$, $\dim M'<\infty$ and
\begin{align*}
|\tr (P_{M'}TP_{M'})|
=&\Bigl|\tr (P_{K}TP_{K})+\sum_{j=1}^{n+1}\langle Tx_j,x_j\rangle\Bigr|\\
\le& \frac{\e}{2}+\Bigl|\alpha k+\sum_{j=1}^{n+1}\langle Tx_j,x_j\rangle\Bigr|\\
\le&
\frac{\e}{2}+|\alpha k + n\beta + \gamma|+(n+1)\cdot\frac{\e}{2(n+1)}\\
\le&\e.
\end{align*}
\end{proof}
Now the convexity of ${\mathcal D}_{{\rm const}}(T)$ is a direct consequence of Lemma \ref{induction}.
However, we prove a property of ${\mathcal D}_{{\rm const}}(T)$ slightly stronger than convexity,
which is the main result of this note.
\begin{theorem}\label{convex_diagon}
Let $H$ be a separable Hilbert space, $\dim H=\infty,$ and let $T\in B(H).$ If $\alpha\in{\mathcal D}_{{\rm const}}(T)$ and $\beta\in W_e(T),$ then 
\begin{equation}\label{convex}
t\alpha +(1-t)\beta\in{\mathcal D}_{{\rm const}}(T), \qquad t \in(0,1).
\end{equation}
 Thus, for a Hilbert space $H$ of any dimension and $T\in B(H),$ the set ${\mathcal D}_{{\rm const}}(T)$ is convex.
\end{theorem}

\begin{proof}

To prove \eqref{convex}, without loss of generality, it is sufficient to show that if $\alpha < 0<\beta$, $\alpha\in{\mathcal D}_{{\rm const}}(T)$, and $\beta \in W_e(T),$ then $0\in{\mathcal D}_{{\rm const}}(T)$. Otherwise, we can replace $T$ by a suitable linear combination 
$a T+b I, a, b \in \mathbb C,$ if necessary.

Fix an orthonormal basis $(e_j)_{j=1}^\infty$ in $H$. Set $M_0=\{0\}$.
Using  Lemma \ref{induction} inductively,  
construct the family of finite-dimensional subspaces $M_k\subset H$, $k\ge 1,$ such that
$$
M_{k+1}\supset (M_k\vee e_{k+1}) \qquad \text{and} \qquad |\tr (P_{M_{k+1}}TP_{M_{k+1}})|< (k+1)^{-1}
$$
for all $k \ge 0.$ 
Choose inductively an orthonormal sequence $(u_j)_{j=1}^\infty$ such that $\{u_j: 1 \le j \le \dim M_k \}$ is an orthonormal basis in $M_k.$ 
By construction,
$$
\overline{\bigcup_{k=1}^\infty M_k}=H,
$$
for all $k\in\NN,$ hence $(u_j)_{j=1}^\infty$ is an orthonormal basis in $H.$ Since
$$
\lim_{k \to \infty} \sum_{j=1}^{\dim M_k}\langle T u_j, u_j \rangle=0,
$$
from  Theorem \ref{FanT} it follows that  $0\in{\mathcal D}_{{\rm const}}(T)$.
In particular, the set ${\mathcal D}_{{\rm const}}(T)$ is convex. 

If $\dim H <\infty,$
then the convexity of ${\mathcal D}_{{\rm const}}(T)$ is noted in Corollary \ref{corol}.
\end{proof}
Observe that by \cite[Theorem 2.3.4]{Webster}, if  $K \subset \mathbb C^n$ is convex, then for any $x$ from the interior of $K$
and $y \in \overline{K}$ the points $t x +(1-t)y, t \in (0,1],$ belong to the interior of $K.$
Thus, in view of \eqref{interior}, 
Theorem \ref{convex_diagon}
has new operator-theoretical content only if $\alpha, \beta \in \partial W_e(T).$

For any $n$-tuple $\cT=(T_1,\dots,T_n)\in B(H)^n$  we have clearly $\cD_{{\rm const}}(\cT)\subset W(\cT)$. Next we characterize those $n$-tuples of operators for which the set $\cD_{{\rm const}}(\cT)$ is maximal.

To this aim, recall that a subset $A\subset\RR^k$ is said to be an affine subspace if $M=u+L$ for some $u\in\RR^k$ and a subspace $L\subset \RR^k$. The smallest affine subspace containing a set $A\subset \RR^k$ is called the affine hull of $A$.
A nonempty subset $A\subset\RR^k$ is called relatively open if it is relatively  open in the affine hull of $A$. Denote by $\Relint A$ the relative interior of $A$ in the affine hull of $A$.
The above definitions can be applied also for subsets of $\CC^n$ if we identify $\CC^n$ with $\RR^{2n}$ in the usual way.

For an $n$-tuple $\cT \in B(H)^n$ and a linear mapping $M:\mathbb R^n \to \mathbb R^n$ given by the matrix $(m_{ij})_{i,j=1}^n$  denote by $M\cT$ the $n$-tuple from $B(H)^n$ defined as
$
M\cT:=(\sum_{j=1}^{n}m_{ij}T_j)_{i=1}^n.
$
Below, we identify linear mappings on $\mathbb R^n$ with their matrix representations.
\begin{theorem}\label{bourin}
Let $H$ be separable Hilbert space,  $\dim H=\infty,$ and let $\cT=(T_1,\dots,T_n)\in B(H)^n$. Then the following conditions are equivalent:
\begin{itemize}
\item[(i)] $W(\cT)=\cD_{{\rm const}}(\cT)$;
\item[(ii)] $W(\cT)=\Relint W_e(\cT)$;
\item[(iii)] $W(\cT)$ is convex and relatively open.
\end{itemize}
\end{theorem}

\begin{proof}

(i)$\Rightarrow$(ii): \,
First, we show that $W(\cT)$ is relatively open.

The proof relies on several convenient reductions of the general set-up.

Instead of the $n$-tuple $(T_1,\dots,T_n)$ we may consider the $(2n)$-tuple  of selfadjoint operators  $(\Re T_1,\Im T_1,\dots,\Re T_n, \Im T_n)$ and deal with its joint numerical and essential numerical ranges contained in $\mathbb R^{2n}.$ 
As far as, we are concerned with the \emph{relative} interior of $W_e(\cT),$ without loss of generality, we may assume that $\cT=(T_1,\dots,T_k)$ is a $k$-tuple of selfadjoint operators such that $\cD_{{\rm const}}(\cT)=W(\cT)$.

Suppose on the contrary that there exists $\la=(\la_1,\dots,\la_k)\in W(\cT)\setminus\Relint W(\cT)$.
We may assume that $\la=(0, \dots, 0):=(0)_k.$ (If not, then replace $\cT$ by the $k$-tuple of operators $\cT-\la=(T_1-\la_1, \dots, T_k-\la_k)$).

Thus we consider a $k$-tuple of selfadjoint operators $\cT=(T_1,\dots,T_k)\in B(H)^k$  such that $\cD_{{\rm const}}(\cT)=W(\cT)$ and $(0)_k\in W(\cT)\setminus\Relint W(\cT)$.

After a relabeling, if necessary, we may assume that there is $m, 1\le m\le k,$ such that the operators $T_1,\dots,T_m$ are linearly independent and each $T_j, m<j\le k,$ is a linear combination of $T_1,\dots,T_m$.
Let $\cT^0=(T_1,\dots,T_m,0,\dots,0)$. Note that there exists an invertible linear mapping $M:\RR^k\to\RR^k$ such that $\cT^0=M\cT.$ Hence
$W(\cT^0)=M W(\cT)$, $\Relint W(\cT^0)=M \Relint W(\cT)$, and $\cD_{{\rm const}}(\cT^0)=M \cD_{{\rm const}}(\cT).$  So $\cT^0$ satisfies the same properties as $\cT$: $W(\cT^0)=\cD_{{\rm const}}(\cT^0)$ and $(0)_k\in W(\cT^0)\setminus\Relint W(\cT^0)$.

Denote the truncated $m$-tuple $(T_1,\dots, T_m)$ by $T^0_t$. We have $W(\mathcal T^0_t)=\cD_{{\rm const}}(\mathcal T^0_t)$
and $(0)_m\in W(\mathcal T^0_t)\setminus\Relint W(\mathcal T^0_t)$.
Note that  $\overline{W(\mathcal T^0_t)}=\overline{\cD_{{\rm const}}(\mathcal T^0_t)}\subset W_e(\mathcal T^0_t)$. So 
$\overline{W(\mathcal T^0_t)}=W_e(\mathcal T^0_t)$, and then $\overline{W(\mathcal T^0_t)}$ is convex.
Since $\overline{W(\mathcal T^0_t)}$ is convex and $(0)_m$ lies on its boundary, there is a supporting hyperplane of $\overline{W(\mathcal T^0_t)}$ passing through $(0)_m.$  Hence after a rotation of $\RR^m,$ realized by an orthogonal mapping $U:\RR^m\to \RR^m,$ we can obtain an $m$-tuple $\mathcal S=U \mathcal T^0_t= (S_1,\dots,
S_m)$ of linearly independent selfadjoint operators such that 
$(0)_m\in W(\mathcal S)=\cD_{{\rm const}}(\mathcal S)$ and
$$
W(\mathcal S)\subset\{(r_1,\dots,r_m)\in\RR^m: r_m\ge 0\}.
$$
Therefore, $S_m \ge 0$ and $0\in\cD_{{\rm const}}(S_m)$. Then $S_m=0$, a contradiction with the assumption that the operators $S_1,\dots,S_m$ are linearly independent.

Hence $W(\cT)$ is relatively open.
We have $\overline{W(\mathcal T)}=\overline{\cD_{{\rm const}}(\mathcal T)}\subset W_e(\mathcal T)$. So 
$\overline{W(\mathcal T)}=W_e(\mathcal T)$ and
$W(\cT)=\Relint\overline{W(\cT)}=\Relint W_e(\cT)$.
  \smallskip

(ii)$\Rightarrow$(iii): Clear.
\smallskip

(iii)$\Rightarrow$(i):\,
We show that $\overline{W(\cT)}=W_e(\cT)$. Suppose on the contrary that $\overline{W(\cT)}\setminus W_e(\cT)\ne\emptyset$. 
 Then there exists $\la$ in the relative topological boundary of $W(\cT)$ such that $\la\notin W_e(\cT)$. Since $W(\cT)$ is relatively open, $\la\notin W(\cT)$ either, and therefore
$\la\in \overline{W(\cT)}\setminus \bigl(W(\cT)\cup W_e(\cT)\bigr)$. By \eqref{lancaster}, we have
$$
\la\in\overline{W(\cT)}\subset\conv \left(W(\cT)\cup W_e(\cT)\right).
$$
Since both $W(\cT)$ and $W_e(\cT)$ are convex, there exist $\mu\in W(\cT)$, $\nu\in W_e(\cT)$, and $t\in(0,1)$ such that
$$
\la=t\mu+(1-t)\nu.
$$
Let $L$ be the affine hull of $W(\cT)$. Since $W(\cT)$ is relatively open in $L$, there exists $\e>0$ such that $\mu'\in W(\cT)$ whenever $\mu'\in L$ and $|\mu'-\mu|<\e$.

We have $\nu\in W_e(\cT)\subset \overline{W(\cT)}$. So there exists $\nu'\in W(\cT)$ such that $|\nu'-\nu|<\e t$.
Let 
$$
\mu'=\mu -\frac{\nu'-\nu}{t}(1-t).
$$ 
Then $\mu'\in W(\cT)$ and 
$$
\la=  t\mu'+ (1-t)\nu'.
$$
Hence $\la$ is a convex combination of elements of $W(\cT)$, and so  $\la\in W(\cT)$, a contradiction 

Thus $\overline{W(\cT)}=W_e(\cT).$ Since $W(\cT)$ is relatively open, 
we infer that 
\begin{equation*}\label{inclusion}
W(\cT) = \Relint W_e(\cT)\subset \cD_{{\rm const}}(\cT)\subset W(\cT)
\end{equation*}
by \cite[Corollary 4.2]{MT} (cf. \eqref{interior} and comments following it).
Therefore, $\cD_{{\rm const}}(\cT)=W(\cT)$.

\end{proof}
\begin{remark}\label{bourin1}
If $\dim H <\infty,$ then $\Relint W_e(\cT)=\emptyset,$ and Theorem \ref{bourin} becomes false in this case. However, in view Corollary \ref{corol},  we can still claim that (i)$\Leftrightarrow$(iii) by trivial reasons.
\end{remark}

The examples of $n$-tuples of operators $\cT$ with convex and relatively open $W(\cT)$ are, in particular, provided by $n$-tuples of Toeplitz operators on the Hardy space 
$H^2(\mathbb D),$ where $\mathbb D$ is the unit disc, see \cite[Proposition 3]{Cho} (and also \cite{Klein} for $n=1$).
In fact, in this case $W(\cT)$ is open unless it is a single point.

Taking account Theorem \ref{bourin} and Remark \ref{bourin1}, we get the following corollary for single $T \in B(H).$ Note that it was stated in \cite[Proposition 1.4]{B} without proof.

\begin{corollary}\label{Bour}
Let $H$ be a separable Hilbert space, and let $T\in B(H)$. Then ${\mathcal D}_{{\rm const}}(T)=W(T)$ if and only if $W(T)$ is relatively open.
\end{corollary}

Apparently, the simplest example of $T\in B(H)$ with open $W(T)$ is provided by selfadjoint $T$ such that 
$m:=\min \sigma(T)$ and $M:=\max\sigma(T)$ are not eigenvalues of $T$ and $M>m.$ Indeed, it is well-known that in this case $m$ and $M$ do not belong to $W(T).$ Since  $\overline {W(T)}=\conv \sigma(T)=[m, M]$ and  $W(T)$ is an interval, 
it follows that $W(T)=(m, M).$
Apart from Toeplitz operators mentioned above, the examples of $T\in B(H)$ with  open $W(T)$  include, in particular,
 weighted shifts with periodic weights, see \cite[Proposition 6]{Stout}. Several more general classes of weighted shifts 
with open numerical ranges were described in \cite{Tam} and \cite{Wang}. Remark that the numerical range of a weighted shift
is an open or closed disc centered at the origin, so this class of operators fits very well into the framework of Corollary \ref{Bour}.
Unfortunately, the numerical ranges of tuples of weighted shifts have not been studied in the literature.

\begin{remark}
Note that if $T\in B(H)$ and  $\lambda \in {\mathcal D}_{{\rm const}}(T),$ then $\lambda$ is an element of $W(T)$ which is attained on a set spanning $H.$
Thus from \cite[Theorem 1]{Embry} (or from the argument given in the proof of (i)$\Rightarrow$(ii), Theorem \ref{bourin}), it follows that ${\mathcal D}_{{\rm const}}(T)\subset \Relint W(T)$ for any $T \in B(H).$ 
So $\Relint W(T)$ arises naturally in the study of ${\mathcal D}_{{\rm const}}(T)$ even when  $W(T)$ is not relatively open. 
\end{remark}

Corollary \ref{Bour} describes the situation when $\cD_{{\rm const}}(T)$ is maximal. On the other hand, ${\mathcal D}_{{\rm const}}(T)$ can be empty.

\begin{example}
Let $H=\ell_2(\mathbb N)$ and let $T \in B(H)$ be the diagonal operator given by $T:=\diag(1,\frac{1}{2},\frac{1}{4},\dots)$. Then $T$ is compact and ${\mathcal D}_{{\rm const}}(T)\subset W_e(T)=\{0\}$. However, $0\notin W(T)$, so ${\mathcal D}_{{\rm const}}(T)=\emptyset$.
\end{example}

\section{Final remarks}\label{diagonals}

The convexity of $\mD_{{\rm const}}(T_1,\dots,T_n)$ for $n\ge 2$ is an open problem, even for commuting operator tuples. However, the next example shows that Fan's Theorem 
\ref{FanT} is not true for $n$-tuples of operators.

\begin{example}
Let $n=2$ and $T_1,T_2$ be the operators on $\CC^2$ considered in Example \ref{e}.
Let $\{e_1, e_2\}$ be the standard basis in $\CC^2$. Let $F$ be the separable infinite-dimensional Hilbert space with an orthonormal basis $(e_j)_{j=3}^\infty$. Let $H=\mathbb C^2\oplus F$ and let $S_1,S_2\in B(H)$ be defined by $S_1=T_1\oplus 0_{F}$ and
$S_2=T_2\oplus 0_{F}$. Then
 $$\sum_{j=1}^k\langle S_1e_j,e_j\rangle=\sum_{j=1}^k\langle S_2e_j,e_j\rangle=0$$
 for all $k\ge 2$, but $(0,0)\notin\mD_{{\rm const}}(S_1,S_2)$. Indeed, let us show that if $$\langle S_1 h,h
\rangle=\langle S_2h,h\rangle=0$$ for a unit vector $h=(x,f)\in H,$ then $x=0$.

Suppose on the contrary that $x \neq 0.$
Then
$$
\bigl(\langle S_1h,h\rangle,\langle S_2h,h\rangle\bigr)\in\|x\|^2_{\mathbb C^2} \cdot W(T_1,T_2),
$$
but $(0,0)\notin \|x\|^2_{\mathbb C^2} \cdot W(T_1,T_2)$ in view of Example \ref{e}, a contradiction.
Hence $(0,0)\notin \mD_{{\rm const}}(S_1,S_2)$.
\end{example}

\begin{remark}
The pair of operators $(S_1,S_2)$ in the previous example can be identified with the triple of selfadjoint operators 
$(\Re S_1,\Im S_1, S_2)$. Thus the example shows that Theorem \ref{FanT} is not true even for triples of selfadjoint operators (in spite of the fact that the joint numerical range of any triple of selfadjoint operators on a Hilbert space of dimension at least $3$ is convex, see e.g. 
\cite[Theorem 1]{Feintuch} and \cite[Theorem 5.4]{gutkin}).
\end{remark}

Naturally, given $T \in B(H)$ with $\dim H=N$, $1\le N\le\infty$, one may attempt to study the convexity of the set $\mathcal D (\mathcal T)$ of all diagonals of $T,$ i.e.,   
the convexity of the subset of $\ell^\infty$ given by
$$
\mathcal D (T):=\left \{(\langle T e_n, e_n\rangle)_{n=1}^N: \,\, (e_n)_{n=1}^N \, \text{is an orthonormal basis of} \, \,H\right \}.
$$
Note that since the unitary group in $B(H)$ is path connected (\cite{Cordes}), the set $\mathcal D(T)$ is path-connected as well.
However, this direction seems to be much more demanding,  at least in our general setting.
If $N<\infty$ then $\mathcal D (T)$ coincides with the ``$N$-dimensional'' numerical range $\mathcal W(T)$ defined in \cite{FiWi71}.
It was noted in \cite{FiWi71} that while $\mathcal D(T)$ is convex for selfadjoint $T$  (by an old result due to Horn),
the convexity of $\mathcal W(T)$ may fail if  $T$ is normal. Later on, it was proved in \cite{AS}
that $\mathcal W(T)$ is convex if and only if there exist $\alpha \in \mathbb C, \alpha \neq 0,$ and $\beta \in \mathbb C$ such that $\alpha T +\beta$  is selfadjoint, implying that $\mathcal W (T)$ is not convex for most of normal $T.$

If $N =\infty,$ then it was discovered in \cite{Neum} that if $T\in B(H)$ is selfadjoint, then the $\ell^\infty$-closure of $\mathcal D(T)$ is convex. This result may lead to a hope that $\mathcal D(T)$  
is convex for such a $T,$ as in the case $\dim H<\infty.$ Slightly later, R. Kadison proved  in \cite{Kadison02a, Kadison02b}
that a sequence $d=(d_n)_{n=1}^\infty$ is
a diagonal of some \emph{selfadjoint} projection in $H$ if and only if it takes values in $[0, 1]$ and if
the sums $a(d) := \sum_{d_j < 1/2} d_j$ and $b(d):=\sum_{d_j \ge 1/2} (1-d_j)$ satisfy either $a(d)+b(d)=\infty$, or $a(d)+b(d) <\infty$ and $a(d)-b(d) \in \mathbb Z.$
Using this description of $\mathcal D(P)$, it is easy to show 
that $\mathcal D(P)$ is not, in general, convex even in this, comparatively simple case. 
(See \cite{Arveson06} and \cite{Arveson_USA} for more details on Kadison's result and its improvements by W. Arveson). While 
Kadison's theorem concerns the set $$\bigcup \{\mathcal D(P): \, P \,\text{is a selfadjoint projection on}\,\, H\},$$
 it is easy to adopt it to our framework of fixed $P$ (as observed in \cite[p. 94]{Loreaux16}). It suffices to note that the selfadjoint projections $P$ and $Q$ are unitary equivalent 
 if and only if 
$ \tr P=\tr Q$ and $\tr(I-P) =\tr (I-Q),$ and for $(d_k)_{k=1}^{\infty}$ as above, $\tr P=\sum_{k \ge 1} d_k$ and $\tr(I-P)=\sum_{k \ge 1} (1-d_k).$

If $d^1=(0,1,0, 1, \dots)$ and $d^2=(1/2, 1, 1/4, 1, 1/8, \dots),$ then both sequences can be realized as diagonals of the same projection $P_0,$
since $a(d^1)=b(d^1)=0,$ and $a(d^2)=1, b(d^2)=0,$ and the corresponding traces are infinite.
On the other hand, for $d^0=(d^1+d^2)/2=(1/4, 1, 1/8, 1, \dots),$ one has $a(d^0)=1/2$ and $b(d^0)=0.$
Hence $d^0$ is not a diagonal of a projection by Kadison's theorem, and $\mathcal D (P_0)$ is not convex.  
A version of this example has already appeared in \cite[Example 3.0.1]{L}, but we feel that the details given above would nicely supplement our discussion here.

Despite the convexity of $\mathcal D(T)$ may, in general, fail even for selfadjoint $T,$ it was proved that compact positive operators, 
which are either of finite rank or of infinite rank and with infinite-dimensional kernel, have convex sets of diagonals. See \cite[Corollary 6.7]{KW} and \cite[Corollary 4.3]{LW15} for these results.

\section{Acknowledgment}
The authors would like to thank the anonymous referees for very careful reading the manuscript and many useful comments that improved the presentation considerably.

\end{document}